\documentclass[a4paper,12pt]{article}
\usepackage{graphicx,amssymb,amstext,amsmath}

\usepackage[latin1]{inputenc}
\usepackage{amsthm}
\usepackage{subfig}
\usepackage{cancel} 
\usepackage{dsfont}
\usepackage{amsmath}
\usepackage{amsfonts}
\usepackage{amssymb}
\usepackage{mathrsfs}
\usepackage{graphicx}
\usepackage{graphicx,color}
\usepackage[all]{xy}
\usepackage[total={5.2in,9.1in},
top=1.4in, left=1.55in, includefoot]{geometry}

\newtheorem{R}{Remark}

\newtheorem{T}{Theorem}[section]
\newtheorem{Pro}{Proposition}
\newtheorem{C}{Corollary}

\newtheorem{Le}{Lemma}[section]

\newcommand{\ds}{\displaystyle}

\newcommand{\re}{\mathbb{R}}

\title{\textbf{Contributions to the study of Anosov Geodesic Flows in Non-Compact Manifolds}}
\author{\'Italo  Melo  and Sergio Roma\~na  }  
\date{}

\begin{document}
\maketitle

\begin{abstract}
In this paper we prove that if the geodesic flow of a {compact or non-compact} complete manifold without conjugate points is of the Anosov 
type, then the average of the integral of the sectional curvature along the geodesic is negative and away from zero from a uniform time. Moreover, in dimension two, if the manifold has no focal points, then this condition is sufficient to obtain that the geodesic flow is of
Anosov type. This sufficient condition will also be used to construct new examples of non-compact surfaces whose geodesic flow is of the Anosov type.
\end{abstract}
 \section{Introduction}

In \cite{Ano:69}, Anosov proved that geodesic flows of manifold of negative curvature provides chaotic dynamical systems, that is, the degree of complexity of the orbits is very high.  These geodesic flow  are called uniformly hyperbolic systems or simply ``Anosov" systems.  When the manifold is compact, these geodesic flow are stably ergodic, mixing and decay of correlations. When the manifold  is not compact, but its curvature is negative pinched (bounded between two negative constants),  the same argument by Anosov implies that the geodesic flow is also Anosov (cf. \cite{Kni:02}).  Some important questions:
\begin{enumerate}
 \item What condition in the curvature implies Anosov condition for the geodesic flow? 
 \item What geometric conditions impose the Anosov condition?
\end{enumerate}
The question 2 is very generic, because it may depend on which geometric invariant we are looking for.\\ 
\indent Some answers to the question 2, were obtained by many authors, by example, \cite{kli:74} showed, for compact manifold, the Anosov condition implies a geometry without conjugate points, this result was generalized by Ma\~n\'e in \cite{man:87} for manifold of finite volume. In \cite{Ruggiero1991}, Ruggiero showed that the $C^2$-interior of metrics  without conjugate points are precisely the Anosov metrics. \\
\indent The above results show the intimate relationship between geometry and dinamics of Anosov geodesic flows. \\
Concern to the question 1, for many time, it was thought that negative curvature was a necessary condition to the geodesic flow to be Anosov. But, in {\cite{Ebe:73}, Eberlein constructed examples of compact manifolds with non-positive curvature whose  geodesic flow is Anosov and that has open subset where the sectional curvature is zero on all tangent planes, in other words, negative curvature does not is a necessary condition for a geodesic flow to be Anosov.  Moreover, in 2003, Donnay and Pugh, (see \cite{DP:03}) constructed a example of a compact surface isometrically embedded in $\mathbb{R}^3$ whose geodesic flow is Anosov, in particular, this surface has points with positive curvature.\\
The above examples shown that all sign of the curvature can be appear and still the geodesic flow can be Anosov.\\
\indent The more general equivalences, concern the question 1, in compact manifold  or compactly homogeneous (the isometry group of its universal cover acts co-compactly) was given by  Eberlien  {\cite{Ebe:73}}. In this paper, he showed, in particular that, whenever the geodesic flow is Anosov, then  the negative sign of the curvature appears at some point throughout the geodesic.  For more general case, in \cite{Bol:79} Bolton, proved the same equivalences proved by Eberlein for non compact manifold. Both results does  not say nothing about the quantity of negative curvature that appear, when the geodesic flow is Anosov.\\
\indent This paper focuses on the treatment of the geometric conditions, compact of non-compcat manifold, imposed by the Anosov condition, \emph{i.e.}, we obtain an answer to the question 2. This answer is related to the intergral of curvature.  Moreover, related to the question 1,  we get a geometric condition that imply the condition of Anosov for non-compact surface. \\
To announce the results of this work, we begin with the formal definition of the Anosov geodesic flow.

\indent Let $(M, \langle, \, \rangle)$ be a complete Riemannian manifold and $SM$ the unitary tangent bundle. 	
Let $\phi^t:SM \rightarrow SM$ be the geodesic flow and suppose that $\phi^t$ is Anosov. This means that  $T(SM)$ have a 
splitting $T(SM) = E^s \oplus \langle G \rangle \oplus E^u $ such that 
\begin{eqnarray*}
	d\phi^t_{\theta} (E^s(\theta)) &=& E^s(\phi^t(\theta)),\\
	d\phi^t_{\theta} (E^u(\theta)) &=& E^u(\phi^t(\theta)),\\
	||d\phi^t_{\theta}\big{|}_{E^s}|| &\leq& C \lambda^{t},\\
	||d\phi^{-t}_{\theta}\big{|}_{E^u}|| &\leq& C \lambda^{t},\\
\end{eqnarray*}
for all $t\geq 0$ with $ C > 0$ and $0 < \lambda <1$, where  $G$ is the vector field derivative of the geodesic vector flow. 


\indent For any $\theta=(p,v)\in SM$, we will denoted by $\gamma_{\theta}(t)$ the unique geodesic with initial conditions
$\gamma_{\theta}(0)=p$ and $\gamma'_{\theta}(0)=v$. Let $V(t)$ be a nonzero unit and perpendicular  parallel vector field along 
$\gamma_{\theta}(t)$. We denote by $k(V(t))$ the sectional curvature of the subspace spanned by $\gamma'_{\theta}(t)$ and $ V(t)$, if $M$ is a surface then we will denote by $k(\gamma_{\theta}(s))$ the sectional curvature in the point $\gamma_{\theta}(s)$.\\

Now we present our first result,
\begin{T}\label{main}
Let $M$ be a complete manifold with curvature bounded below by $-c^2$ without conjugate points and whose geodesic flow is Anosov. Then, there are two positive constants $B$ and $t_0$  such that, for all $\theta \in SM$ and for any unit perpendicular parallel vector field $V(t)$ along  $\gamma_{\theta}(t)$ we have that 
\begin{equation}\label{E1main}
\displaystyle\frac{1}{t}\displaystyle\int_{0}^{t} k(V(r)) \ dr   \leq -B,
\end{equation}
whenever if $t > t_0$.
\end{T}

When the manifold $M$ is compact or compactly homogeneous, \emph{i.e.}, the isometry group of its universal cover acts co-compactly,  Eberlien in {\cite{Ebe:73}} (see Corollary 3.4 and Corollary 3.5) proved the following result:\
\ \\
\ \\
\textbf{Theorem}\, {\cite{Ebe:73}}\emph{ Assume that  $M$   has no conjugate points, then 
\begin{enumerate}
\item[\emph{(i)}]If the geodesic flow is Anosov, then for all $\theta\in SM$ and for any  nonzero perpendicular  parallel vector field $V(t)$ along $\gamma_{\theta}(t)$ there is $t$ such that $k(V(t))<0$.
\item[\emph{(ii)}] If $M$ has no focal points and satisfies the condition of \emph{(i)}, then geodesic flow is of  Anosov type.
\end{enumerate}
}
Thus, Theorem \ref{main}  generalizes the result  {(i)} of the above theorem. It is worth emphasizing that this result can be applied in
non-compact manifolds.  \\
\ \\
Some immediate consequences of Theorem \ref{main} are:
\begin{C}\label{Cor1}
Let $M$ be a complete manifold with curvature bounded below by $-c^2$, whose geodesic flow is Anosov. Then, if $M$ has finite volume, we get 
$$\int_{SM}\emph{Ric}\, d\mu<-B\cdot\mu(SM)<0,$$
where \emph{Ric} is the Ricci curvature and $\mu$ is the Liouville 	measure on $SM$. 
\end{C}
As the Gauss-Bonnet theorem holds for surfaces of finite volume (see {\cite{Ros:82}}), then the above corollary implies that, neither surface with Euler Characteristic  zero or postive and with finite volume admits a  geodesic flow of Anosov type.  
\begin{C}\label{Cor2}
No manifold $M$ that admits a geodesic $\gamma(t)$ and nonzero perpendicular para-\\llel vector field $V(t)$ with $k(V(t))\geq 0$ for any $t \geq t_1$, has geodesic flow of Anosov type. \\
\indent In particular, if $M$ is the product of two manifolds with curvature bounded below, furnished with the product metric, then the geodesic flow is never an Anosov flow. 
\end{C}
 It is worth noting  that, the second part of this corollary was well known for product\newline of compact manifolds or  product of compactly
 homogeneous manifolds (cf. \cite{Ebe:73}). Therefore, ours is a more general result. \\ 
 \indent  Thus, when we have the product of two manifolds, the last  result leads us to think that if it is possible to {change the product
 metric in such a way that the geodesic flow becomes  Anosov. In fact, in the Section \ref{Sec_Examples}, using the Theorem \ref{main2} and a 
 ``Warped Product" to construct a metric in $\re \times \mathbb{S}^1$ whose geodesic flow is Anosov (see Section  \ref{Sec_Examples}). \\
\ \\
\indent Our third corollary stated that the world of the compact manifold and non-compact manifold are, in some way, very different.\\
Before we present the corollary, we consider the following function $\mathcal{K}(t)$, associated to the sectional curvatures of the a complete
non-compact manifold $M$, as the follows:\\

For each $x\in M$ and each plane $P\subset T_x M$, we denotes by $k(P)$ the sectional curvature of plane $P$. Thus, we defines $k(x)=\ds \sup_{P\subset T_{x}M}k(P)$.  Fixed a point $O\in M$, we {define}  $$\mathcal{K}(t):=\ds \sup_{x \in M\setminus B_{t}(O)}k(x),$$
where $B_{t}(0)$ is the open ball of center $O$ and radius $t$.\\
\indent We say that a complete non-compact manifold $M$ is \emph{asymptotically flat} if \\ $\ds \lim_{t\to +\infty}\mathcal{K}(t)=0$.
\begin{C}\label{Cor3} 
Let $M$ be a asymptotically flat manifold, assume that $M$ has no conjugate points and curvature bounded below. Then its geodesic flow is not an
Anosov flow. 
\end{C}

This corollary allows us to construct a large category of non-compact manifold whose geodesic flow is not Anosov. By example, any complete minimal surface embedded in $\re^3$ with finite total curvature is asymptotically flat (see \cite{Sch:83}). \\
In particular, the geodesic flow of the minimal  surface as catenoid, helicoid, Costa's surface and Costa-Hoffman-Meeks surface is not Anosov (see some picture below).

\begin{figure}[!htb]
	\centering
	\subfloat[Catenoid]{
		\includegraphics[height=4cm]{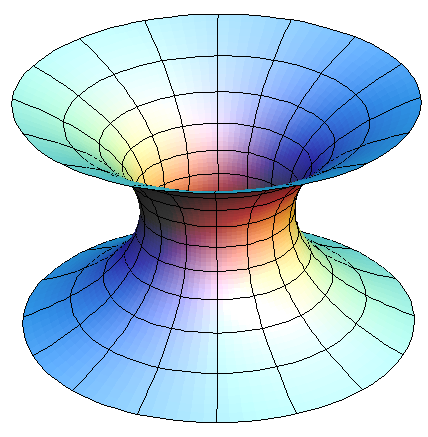}
		\label{figdroopy}
	}
	\quad 
		\subfloat[Helicoid]{
		\includegraphics[height=4cm]{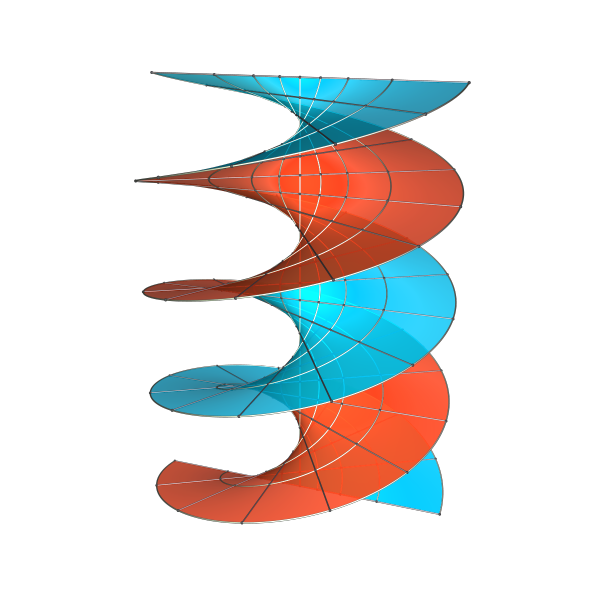}
		\label{helicoid}
	}
	\quad 
	\subfloat[Costa's surface]{
		\includegraphics[height=4cm]{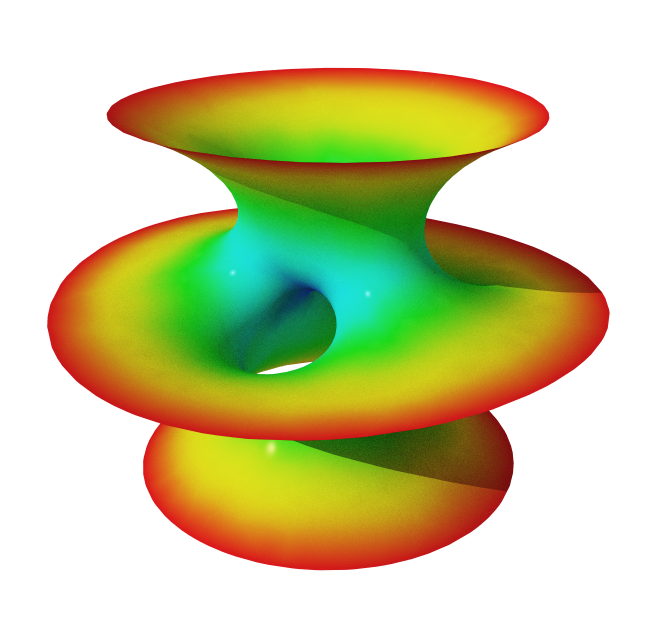}
		\label{figsnoop}
	}
	\label{fig01}
\end{figure}

Note that for manifold of negative curvature bounded between two negative constants (this condition is called ``\emph{pinched}"), its geodesic 
flow is Anosov (cf. \cite{Ano:69}). However, in the non-compact case, the negative sign  of the curvature does not implies that we have a 
geodesic flow of Anosov type.  In fact, in Section \ref{Sec_Examples}, we construct a non-compact surface of negative curvature whose 
geodesic is not Anosov. In other words, from the point view of the dynamic, a manifold of negative curvature and a manifold of pinched negative 
curvature are totally different.   
 

The second result of this paper is a more general version, in dimension two, of (ii) at Theorem \cite{Ebe:73} above. In fact, if our manifold 
has dimension two and does not have focal points, then the conclusion of the Theorem \ref{main} is a sufficient 
condition to the geodesic flow to be an Anosov flow. More precisely,  

\begin{T}\label{main2}
Let $M$ be a complete surface with curvature bounded below by $-c^2$ without focal points. Assume that, there are two constants $B, t_0>0$ such that for all geodesic $\gamma(t)$,
\begin{equation}\label{e1main2}
\frac{1}{t}\int_{0}^{t} k(\gamma(s)) \ ds   \leq -B \ \ \text{whenever} \ \ t > t_0,
\end{equation} 
then the geodesic flow is an Anosov flow.
\end{T}

In the Subsection \ref{Examples}, we use the Theorem \ref{main2} to construct a family of non-compact surfaces with non-positive curvature, which is non compactly homogeneous, in particular, the Eberlein result does not apply, but its  geodesic flow is Anosov. Using this family of surfaces, we construct a non compact surface with regions of positive curvature and Anosov geodesic flow.  To be more specific, we make a ``\emph{Warped Product}" of the circle $S^1$ and  the line $\mathbb{R}$ to construct such surfaces (cf. Section \ref{Sec_Examples}). \\

We take Theorem \ref{main} and Theorem \ref{main2} to present the following corollary,
\begin{C}\label{Cor3'}
Let $M$ be a complete surface with curvature bounded below by $-c^2$ without focal points. Then the geodesic flow is Anosov if and only if there
are two constants $B, t_0>0$ such that for all geodesic $\gamma(t)$
\begin{equation*}
\frac{1}{t}\int_{0}^{t} k(\gamma(s)) \ ds   \leq -B \ \ \text{whenever} \ \ t > t_0.
\end{equation*}
\end{C}
\ \\

We believe that, in greater dimension, it is possible to obtain a similar result of Theorem \ref{main2}. In fact: \\
Suppose that $M$ has dimension $n$. Then consider for each $\theta\in SM$ consider a orthonormal family of $n$ perpendicular parallel vector fields $V_{j}(t)$ along 
$\gamma_{\theta}$, $j=1,\dots, n$, where $V_n(t) = \gamma_{\theta}'(t)$. Then, we present the following conjecture, 

\textbf{Conjecture:}\, \emph{Let $M$ be a complete  manifold with curvature bounded below, without focal points. Assume that, there are two positive constants  $B,	t_0$ such that for all geodesic $\gamma_{\theta}(t)$ and all $j=1,\dots, n-1$ we have 
$$\frac{1}{t}\int_{0}^{t}k(V_{j}(s))ds\leq -B\,\, \text{whenever} \,\, t>t_0,$$
then the geodesic flow is an Anosov flow.} \\
\ \\
\textbf{Structure of the Paper:}\, In Section \ref{Not_Pre}, we present the notation and geometry setting of the paper.
In Section \ref{Sec_Main1}, we present the proof of Theorem \ref{main} and its consequences. In Section \ref{Nec_Cond}, we present the proof of
Theorem \ref{main2} and some results addressed to the conjecture. Finally, in Section \ref{Sec_Examples}, we apply the Theorem \ref{main2} to
construct a new family of  non-compact surfaces with Anosov geodesic flow. At the end of Section \ref{Sec_Examples}, we also  show a example of {a}
non-compact surface of negative curvature whose geodesic flow is not Anosov. 
\section{Notation and Preliminaries}\label{Not_Pre}
Through rest of this paper, $M=(M,g)$ will denote a complete Riemannian manifold without boundary of dimension $n\geq 2$, $TM$  its the tangent bundle, $SM$  its unit tangent bundle, $\pi\colon TM\to M$ will denote the canonical projection, and $\mu$ the Liouville measure of $SM$ (see \cite{P}). 

\subsection{Geodesic flow }\label{GeoFlow}
For $\theta=(p,v)$ a point of $SM$. Let $\gamma_{\theta}(t)$ be the unique geodesic with initial conditions $\gamma_{\theta}(0)=p$ and 
$\gamma_{\theta}'(0)=v$. For a given $t\in \mathbb{R}$, let $\phi^t:SM \to SM$ be the diffeomorphism given by 
$\phi^t(\theta)=(\gamma_{\theta}(t),\gamma_{\theta}'(t))$. Recall that this family is a flow (called the \textit{geodesic flow}) in the sense that  $\phi^{t+s}=\phi^{t}\circ \phi^{s}$ for all $t,s\in \mathbb{R}$. 

Let $V:=ker\, D\pi$ be the \textit{vertical} sub-bundle of $T(TM)$ (tangent bundle of $TM$). \\
Let $\alpha\colon TTM\to TM$ be the Levi-Civita connection map of $M$. Let $H:=ker \alpha$ be the horizontal sub-bundle. Recall that, $\alpha$ is defined as follow: Let $\xi \in T_{\theta}TM$ and $z:(-\epsilon,\epsilon) \rightarrow TM$  be a curve adapted to $\xi$, \emph{i.e.},  $z(0) = \theta$ and $z'(0) = \xi$, where $z(t) = (\alpha(t),Z(t))$, then
$$\alpha_{\theta}(\xi)=\nabla_{\frac{\partial}{\partial\,t}}Z(t)|_{t=0}.$$


{For each $\theta$, the maps $d_{\theta}\pi|_{H(\theta)}: H(\theta) \rightarrow T_pM$ and $K_{\theta}|_{V(\theta)}:V(\theta) \rightarrow T_pM$ are linear isomorphisms. Furthermore, $T_{\theta}TM = H(\theta) \oplus V(\theta)$ and the map  $j_{\theta}:T_{\theta}TM \rightarrow T_pM \times T_pM$ given by 
	$$ j_{\theta}(\xi) = (D_{\theta}\pi(\xi),K_{\theta}(\xi))        $$
	is a linear isomorphism. }

{	Using the decomposition $T_{\theta}TM = H(\theta) \oplus V(\theta)$, we can identify a vector in $\xi \in T_{\theta} TM$ with the pair of vectors in $T_{p}M$, $(D_{\theta}\pi(\xi),K_{\theta}(\xi))$  and define in a natural way  a Riemannian metric on $TM$
that makes $H(\theta)$ and $V(\theta)$ orthogonal. This metric is called the Sasaki metric and is given by
	$$ g_{\theta}^S(\xi,\eta) = \langle D_{\theta}\pi(\xi), D_{\theta}\pi(\eta)  \rangle  +  \langle K_{\theta}(\xi),    K_{\theta}(\eta)   \rangle. $$}


From now on, we consider the Sasaki metric restricted to the unit tangent bundle $SM$.  It is easy to proof that the geodesic flow preserves
the volume measure  generate by this Riemannian metric in $SM$. Furthermore, this volume measure in $SM$ coincides with the Liouville measure 
$m$ up to a constant. When $M$ has finite volume the Liouville measure is finite.

\noindent Consider the one-form $\beta$ in $SM$ defined for $\theta=(p,v)$ by 
$$\beta_{\theta}(\xi) = g_{\theta}^S(\xi,G(\theta)) = \langle D_{\theta}\pi(\xi),v \rangle_p.$$
Observe that $ker\, \beta_{\theta}\supset V(\theta) \cap T_{\theta} SM $.  It is possible prove that a vector $\xi \in T_{\theta}TM$ lies in $T_{\theta}SM$ 
with $\theta=(p,v)$ if and only if $\langle \alpha_{\theta}(\xi) , v\rangle = 0$. Furthermore, 
$\beta$ is  a contact form invariant by the geodesic flow whose Reeb vector field is the geodesic vector field $G$. Furthermore, the sub-bundle $S=\ker \beta$ is the orthogonal complement of the subspace spanned by $G$. Since $\beta$ is invariant by the geodesic flow, then the sub-bundle $S$ is invariant by $\phi^{t}$, \emph{i.e.}, $\phi^t(S(\theta)) = S(\phi^{t}(\theta))$ for all $\theta\in SM$ and for all $t\in\mathbb{R}$.

To understand the behavior of $d\phi^t$ let us to introduce the definition of Jacobi field.  A vector field $J$ along of a geodesic
$\gamma_{\theta}$ is called the Jacobi field if it satisfies the equation 
\begin{equation}
J'' + R(\gamma'_{\theta},J)\gamma'_{\theta} = 0,
\end{equation}
where $R$ is the Riemann curvature tensor of $M$ and $``\,'\,"$ denotes the covariant derivative along $\gamma_{\theta}$. Note that, for  $\xi=(w_1,w_2) \in T_{\theta}SM$, (the horizontal and vertical decomposition) with $w_1, w_2 \in T_{p} M$ and $\langle v, w_2 \rangle =0$, it is known that  
$d\phi_{\theta}^{t}(\xi) = (J_{\xi}(t), J'_{\xi}(t))$, where $J_{\xi}$ denotes the unique Jacobi vector field along $\gamma_{\theta}$ such that $J_{\xi}(0) = w_1$ and $J'_{\xi}(0) = w_2$. For more details see \cite{P}.

\subsection{No conjugate points}\label{NCP}
Suppose $p$ and $q$ are points on a Riemannian manifold, and $\gamma$ is a geodesic that connects $p$ and $q$. Then $p$ and $q$ are \emph{conjugate points along} $\gamma$  if there exists a non-zero Jacobi field along $\gamma$  that vanishes at $p$ and $q$. 
When neither two  points in $M$ are conjugated, we say the manifold $M$ \emph{has no conjugate points}. Another important kind of manifolds for this paper are the manifolds 
without focal points, we say that a manifold  $M$ \emph{has no focal point}, if for any unit speed geodesic $\gamma$ in $M$  and  for any 
Jacobi vector field $Y$ on $\gamma$  such that $Y(0) = 0$ and $Y'(0)\neq 0$ we have $(||Y||^2)'(t)>0$, for any $t > 0$. It is clear that if
a manifold has no focal points, then also has no conjugate points. \ \\
\indent The more classical example of manifolds without focal points and therefore without conjugate points, are the manifolds of
non-positive curvature. It is possible to construct a manifold having  positive curvature in somewhere, and without conjugate points. There is
many examples of manifold without conjugate points. We emphasize here, for example, in  \cite{man:87}, Ma\~n\'e proved that, when volume is
finite and the geodesic flow is Anosov, then the manifold has no conjugate points. This latter had been proved by Klingenberg
(cf. \cite{kli:74}) in the compact case. In the case of infinite volume the result by Ma\~n\'e \cite{man:87} is an open problem: \\ 
\indent When the geodesic flow is  Anosov, then the manifold has no conjugate points? \\
The last fact showed that, if we would like to work with  geodesic flow of  Anosov type, we assume then that our manifold has no conjugate
points (condition superfluous in finite volume via Ma\~ne result). Therefore, from now on, we can assume that the manifold $M$ has no
conjugate points.\\ 
\indent {Now} suppose that $M$ has no conjugate points and its sectional curvatures are bounded below by $-c^2$. In this case, if the
geodesic flow $\phi^t:SM \to SM$ is Anosov, then in \cite{Bol:79}, Bolton showed  that there exists a positive constant $\delta$ such that for
all $\theta \in SM$, the angle between $E^s(\theta)$ and $E^{u}(\theta)$ is greater than $\delta$. Furthermore, if $J$ is a perpendicular
Jacobi vector field along $\gamma_{\theta}$ such that $J(0) = 0$ then there exists $A>0$ and $s_0\in\mathbb{R}$ such that 
$\parallel J(t) \parallel \geq A\parallel  J(s) \parallel$ for $t \geq s \geq s_0$. Therefore, for  $ \xi \in E^{s}(\theta)$ and 
$\eta \in E^{u}(\theta)$ since $\parallel J_{\xi}(t) \parallel \to 0$ as $t \to + \infty$ and $\parallel J_{\eta}(t) \parallel \to 0$ as 
$t \to - \infty$ follows that $J_{\xi}(0) \neq 0 $ and $J_{\eta}(0) \neq 0 $. In particular, $E^{s}(\theta) \cap V(\theta) = \{0\}$ and 
$E^{u}(\theta) \cap V(\theta) = \{0\}$ for all $\theta \in SM$.\\

\indent For  $\theta = (p, v) \in SM$, we denote by $N(\theta):= \{w \in T_{x}M : \langle w, v\rangle = 0\}$.
By the 
identification of the Subsection \ref{GeoFlow} we can  write $S(\theta):=\emph{ker} \, \beta = N(\theta)\times N(\theta)$, $V(\theta) \cap S(\theta) =  \{0\} \times N(\theta) $ and $H(\theta) \cap S(\theta) =  \{0\} \times N(\theta) $. Thus, if  $E \subset S(\theta)$ is a subspace,
$\dim E = n-1$, and $E \cap (V(\theta) \cap S(\theta)) = \{0\}$ then 
$E \cap (H(\theta) \cap S(\theta))^{\perp} = \{0\}$. Hence, there exists a unique linear map $T: H(\theta) \cap S(\theta) \to  V(\theta) \cap S(\theta)$ such that $E$ is the graph of  $T$. In other words, there exists a unique linear map $T: N(\theta) \to N(\theta)$ such that $E = \{(v,Tv) : v \in N(\theta)\}$. Furthermore, the linear map $T$ is symmetric if and only if $E$ is Lagrangian (see \cite{P}).\\

\indent It is known that if the  geodesic flow is Anosov, then for each $\theta\in SM$, the sub-bundles $E^{s}({\theta})$ and
$E^{u}(\theta)$ are Lagrangian and $E^{s}({\theta}) \oplus E^{u}({\theta}) = S(\theta)$. Therefore, for each $t\in \mathbb{R}$, we can write 
$d\phi^t(E^{s}(\theta)) = E^{s} (\phi^{t}(\theta)) = {\rm graph} \, U_{s}(t)$ and 
$d\phi^t(E^{u}(\theta)) = E^{u} (\phi^{t}(\theta)) = {\rm graph} \, U_{u}(t)$, where 
$ U_{s}(t): N(\phi^{t}(\theta)) \to N(\phi^{t}(\theta))$ and $ U_{u}(t): N(\phi^{t}(\theta)) \to N(\phi^{t}(\theta))$ are symmetric maps.

Now we describe a useful method of L. Green (cf. \cite{Gre:58}), to see what properties the maps $U_{s}(t)$ and $U_{u}(t)$ satisfies.

Let $\gamma_{\theta}$ be a geodesic, and consider $V_1,\mathellipsis,V_n$ a system of parallel orthonormal vector fields  along
$\gamma_{\theta}$ with $V_n(t) = \gamma'_{\theta}(t)$.
If $Z(t)$ is a perpendicular vector field along $\gamma_{\theta}(t)$, we can write $$Z(t) =\displaystyle \sum_{i=1}^{n-1} y_{i}(t)V_i(t).$$ 
Note that $Z(s)$ can be identified with the curve $\alpha(s) = (y_{1}(s), \mathellipsis,y_{n-1}(s))$ and $Z'(s)$ can be identified with the curve $\alpha'(s) = (y_{1}'(s), \mathellipsis,y_{n-1}'(s))$. Conversely, 
any curve in $\mathbb{R}^{n-1}$ can be identified with a perpendicular vector field on $\gamma_{\theta}(t)$, so we can identify $N(\phi^{t}(\theta))$ with $\mathbb{R}^{n-1}$ and consider the maps associated to stable and unstable subbundles defined in $\mathbb{R}^{n-1}$.

Now for each $t \in \mathbb{R}$, consider the symmetric matrix $R(t) = (R_{i,j}(t))$, where $1 \leq i,j \leq n-1$, $R_{i,j} = \langle R(\gamma'_{\theta}(t), V_i(t))\gamma'_{\theta}(t),  V_j(t)) \rangle$ and $R$ is the curvature tensor of $M$. The family of operators $ U_{s}(t): \mathbb{R}^{n-1} \to \mathbb{R}^{n-1}$ and $ U_{u}(t): \mathbb{R}^{n-1} \to \mathbb{R}^{n-1}$ satisfies the Ricatti equation

\begin{eqnarray}\label{Ricatti}
U'(t) + U^{2}(t) + R(t) = 0,
\end{eqnarray}
 see \cite{P}.

Now consider the $(n-1) \times (n-1)$ matrix Jacobi equation
\begin{equation}\label{Jacobi}
Y''(t) + R(t)Y(t) = 0.
\end{equation}
If $Y(t)$ is solution of (\ref{Jacobi}) then for each $x \in \mathbb{R}^{n-1}$, the curve $ \beta(t) = Y(t)x$ corresponds to a Jacobi perpendicular vector on $\gamma_{\theta}(t)$.  For $\theta\in SM$, $r\in \mathbb{R}$, we consider  $Y_{\theta,r}(t)$ be the unique solution of (\ref{Jacobi}) satisfying $Y_{\theta,r}(0) = I$ and $Y_{\theta,r}(r) = 0$. In  \cite{Gre:58}, Green proved  that $\ds\lim_{r \to -\infty}Y_{\theta,r}(t)$ exists for all $\theta\in SM$ (see also \cite[Sect. 2]{Ebe:73}). Moreover, if  we define:
\begin{equation}\label{E1Jac}
Y_{\theta, {u}}(t):= \lim_{r\to -\infty}Y_{\theta,r}(t),
\end{equation}
we obtain a solution of Jacobi equation (\ref{Jacobi}) such that $\det Y_{\theta, {u}}(t)\neq 0$. Furthermore, it is proved in \cite{Gre:58} (see also  \cite{ManFre:82} and \cite{Ebe:73}) that $\ds\frac{DY_{\theta, {u}}}{Dt}(t)=\lim_{r\to -\infty}\frac{DY_{\theta,r}}{Dt}(t)$. Furthermore, if 
$$U_{r}(\theta)=\frac{DY_{\theta,r}}{Dt}(0) ; \, \, U^{u}(\theta)=\frac{DY_{\theta,{u}}}{Dt}(0),$$
then  $$U^{{u}}(\theta)=\lim_{r\to -\infty}U_{r}(\theta).$$
It is easy to proof that (see \cite{ManFre:82})
$$U^{u}(\phi^{t}(\theta))=\frac{DY_{\theta, {u}}}{Dt}(t){Y^{-1}_{\theta,{u}}(}t)$$
for every $t \in \mathbb{R}$. It follows that $U^{u}$ is a symmetric solution of the Ricatti equation 
\begin{eqnarray}
	U'(t) + U^{2}(t) + R(t) = 0,
\end{eqnarray}
see 
analogously, taking the limit when $r\to +\infty$, we have defined $U^s(\theta)$, that also satisfies the Ricatti equation (\ref{Ricatti}).
Furthermore, in \cite{Gre:58}, Green also showed that, in the case of curvature bounded below by $-c^2$,  symmetric solutions of the
Ricatti equation which are defined  for all $t \in \mathbb{R}$ are bounded by $c$, \emph{i.e.},  
\begin{equation}\label{U-bounded}
\displaystyle\sup_{t} \parallel U^{s}(t)  \parallel \leq c \, \, \, \ \ \ \ \ \  \text{and} \, \, \, \ \ \ \ \ \ \displaystyle\sup_{t} \parallel U^{u}(t)  \parallel \leq c.
\end{equation} 



\section{Proof of Theorem \ref{main} and its Consequences}\label{Sec_Main1}
In this section, we prove the Theorem \ref{main} and  the Corollaries \ref{Cor1}, \ref{Cor2} and \ref{Cor3}.

 In this direction, we prove the Lemma \ref{L1-main}, which use a Bolton's  result (see \cite{Bol:79}), to have some control on the sum 
 $||U_{s}||^2+||U_{u}||^2$. More specifically,

\begin{Le}\label{L1-main}
Let $M$ be a complete manifold with curvature bounded below by $-c^2$ without conjugate points and whose geodesic flow is Anosov. Then, 
there is a constant $D > 0$, such that for any $\theta\in SM$ and  for any unit perpendicular parallel vector field $V(t)$ along $\gamma_{\theta}$ we have 
$$||U_{s}(V(t))||^2 + ||U_{u}(V(t))||^2 \geq D, \ \ \text{for all} \ \ t \in \mathbb{R}.$$	  
\end{Le}
\begin{proof}
By Bolton's results (see \cite{Bol:79}) there exists a positive constant $\delta$ such that the angle between $E^s(\phi^{t}(\theta))$ and
$E^{u}(\phi^{t}(\theta))$ is greater than $\delta$. Using the identification defined above, observe that $(V(t), U^s(V(t)) \in E^s(\phi^{t}(\theta))$ and $(V(t), U^u(V(t)) \in E^u(\phi^{t}(\theta))$ (see Subsection \ref{NCP}). Thus, using the definition of the Sasaki metric and  definition of angle, we have for each $t\in \re$
\begin{equation}\label{cos}
    |1 + \langle   U^s(V(t))   , U^u(V(t)) \rangle|       \leq \sqrt{1 +||U_{s}(V(t))||^2}\sqrt{1 +||U_{u}(V(t))||^2 }               \cos \delta.          
\end{equation}
Fix $ 0 < D < 1$ such that $\displaystyle\frac{1-D}{1+D} > \cos \delta$, we claim that   
 $$||U_{s}(V(t))||^2 + ||U_{u}(V(t))||^2 \geq D, \ \ \text{for all} \ \ t\in \re. $$
 In fact, by contradiction, suppose that $||U_{s}(V(t))||^2 + ||U_{u}(V(t))||^2 < D$, then  
\begin{itemize}
\item $||U_{s}(V(t))||^2 < D,$
\item $||U_{u}(V(t))||^2 < D,$ 
\item $||U_{s}(V(t))|| \cdot ||U_{u}(V(t))|| < D$.
\end{itemize}  
Thus, from (\ref{cos})
$$  1 - | \langle   U^s(V(t))   , U^u(V(t)) \rangle |   \leq (1 + D) \cos \delta.                              $$
By Cauchy-Schwarz inequality it follows that
 $$\displaystyle\frac{1-D}{1+D} \leq \cos \delta,$$
which is a contradiction by the choice of $D$. Therefore,
 $$||U_{s}(V(t))||^2 + ||U_{u}(V(t))||^2 \geq D.$$	
\end{proof}
We use the Lemma \ref{L1-main} to show the Theorem \ref{main}. Before, we set the notation use in the Theorem \ref{main}.
\begin{R}
For each $\theta\in SM$ and  for any unit perpendicular parallel vector field $V(t)$ along $\gamma_{\theta}$, we denote by $k(V(t))$ the sectional curvature of the subspace spanned by $\gamma'_{\theta}(t)$ and $ V(t)$.
\end{R}

\begin{proof}[\emph{\textbf{Proof of Theorem \ref{main}}}]
Fix $\theta \in SM$, and consider an unit perpendicular parallel vector field $V(t)$ along $\gamma_{\theta}$. Since the operators $U_{s}$ and $U_{u}$ are symmetric and satisfy the equation (\ref{Ricatti}) follows that
$$   \langle U_{s}(V(t)), V(t) \rangle' +    ||U_{s}(V(t))||^2 + k(V(t)) = 0,                                        $$
$$   \langle U_{u}(V(t)), V(t) \rangle' +    ||U_{u}(V(t))||^2 + k(V(t)) = 0.                                        $$
Integrating the sum of the above equations, we get
\begin{eqnarray*}
      \displaystyle\sum_{*= s,u} (\langle U_{*}(V(t)), V(t) \rangle  -\langle U_{*}(V(0)), V(0) \rangle  ) & +& \displaystyle\int_{0}^{t}  ||U_{s}(V(r))||^2  + ||U_{u}(V(r))||^2  \ dr \\
     &+& 2  \displaystyle\int_{0}^{t}    k(V(r)) \ dr = 0.                                
\end{eqnarray*}
Observe that by {inequality}  (\ref{U-bounded}), $\parallel U_{u}(t)  \parallel \leq c$ and  $\parallel U_{s}(t)  \parallel \leq c$ for all $t \in \mathbb{R}$, which allows to state that 

$$\displaystyle\lim_{t \to + \infty} \displaystyle\frac{1}{t}  \displaystyle\sum_{*= s,u} (\langle U_{*}(V(t)), V(t) \rangle - \langle U_{*}(V(0)), V(0) \rangle  ) = 0.$$
Thus, from the Lemma \ref{L1-main} follows that, there exists $t_0> 0$ such that
$$\displaystyle\frac{1}{t}\displaystyle\int_{0}^{t} k(V(r)) \ dr   \leq -\displaystyle\frac{D}{2}, \ \ \text{whenever} \ \ t > t_0.$$
Taking $B=\frac{D}{2}$, we conclude the proof of theorem.
\end{proof}

\subsection{Consequences of Theorem \ref{main}}
\noindent In the follows, we prove some important consequences of Theorem \ref{main}.\\
\noindent We denote by $\emph{Ric}_{p}(v)$  the Ricci curvature in the direction $v$, which is defined in the follows way: consider  a orthogonal basis of $T_pM$, $\{v,v_1,v_2,\dots, v_{n-1}\}$ , then  $$\emph{Ric}_{p}(v)=\frac{1}{n-1}\sum_{i=1}^{n-1} \langle R(v,v_j)v,v_j \rangle.$$
In other words, $\emph{Ric}_{p}(v)$ is the average of the sectional curvature in planes generates by $v$ and $v_j$.
In particular, as $\emph{R}_{p}(v)$ is the trace of the matrix of curvature, then $\emph{R}_{p}(v)$ does not depend of the orthonormal set
$\{v_1,v_2,\dots, v_{n-1}\}$. Thus, for $(p,v)\in SM$, we denote by  $\emph{Ric}(p,v)=\emph{Ric}_{p}(v)$ the Ricci curvature, which is a
function of $SM$ on the real line. As an immediate consequence of Theorem \ref{main}, we have 

\begin{C}\label{Cor4}
Let $M$ be a complete manifold with curvature bounded below by $-c^2$ without conjugate points and whose geodesic flow is Anosov. Then
there are two constants $B,t_0>0$ such that for any $\theta \in SM$  we have 
\begin{equation}\label{E1Cor4}
\ds\frac{1}{t}\int_{0}^{t} Ric_{\gamma_{\theta}(r)} (\gamma'_{\theta}(r))\ dr   \leq -B,
\end{equation}
whenever $t > t_0$.
\end{C}

Let us to prove Corollary \ref{Cor1} using Corollary \ref{Cor4} and Birkhoff's ergodic theorem.

\begin{proof}[\emph{\textbf{Proof of Corollary \ref{Cor1}}}]
As $M$ is a complete manifold with curvature bounded below by $-c^2$ whose geodesic flow is Anosov, then the condition of finite volume gives us,
thanks to Ma\~n\'e result(cf. \cite{man:87}), that $M$ has no conjugate points. 	Therefore, $M$ is a manifold without conjugate
points and whose negative part of the Ricci curvature in integrable (with respect to Liouville measure) on $SM$, then a Guimar\~aes result
(cf. \cite{Gui:92}) implies that the Ricci curvature is integrable  on $SM$. Moreover, for each $\theta\in SM$ the equation (\ref{E1Cor4})
can be writen as 
$$\ds\frac{1}{t}\int_{0}^{t} Ric({\phi^t(\theta)}) \ dr   \leq -B, \ \text{whenever} \, \, t>t_0.$$
As the Liouville measure, is invariant by the geodesic flow, then the Birkhoff ergodic theorem, applied to the Ricci curvature, provides
us 
$$-c^2\cdot \mu\,(SM)<\int_{SM} Ric \ d \mu<-B\cdot\mu\,(SM)<0,$$
which concludes the proof of Corollary \ref{Cor1}.
\end{proof}
\begin{R} In \emph{\cite{Gui:92}}, Guimar\~aes proved that: If $M$ is a  manifold without conjugate points with the positive or negative part of 
the Ricci curvature integrable, then $$\int_{SM}Ric\, d\mu\leq 0,$$
where the equality holds only if the curvature tensor of $M$    is identically zero. Thus, by \emph{Theorem \ref{main}},  manifold of
zero curvature has no geodesic flow of Anosov type, which implies that in the Anosov case should be $\int_{SM}Ric\, d\mu<0$. Obtaining  
another proof of \emph{Corollary \ref{Cor1}}. 
\end{R}
It is well known that a compact surface $M$ with non-negative Euler characteristic $\chi(M)$, does not admit Riemaniann metric whose geodesic
flow is of Anosov Type.  In fact, if the geodesic flow of a compact surface is Anosov, then  the surface has no conjugate points 
(cf. \cite{kli:74}). Thus, by a Hopf's result (cf. \cite{Hopf:47}) the integral of the Gaussian curvature is non-positive and zero in the case of zero curvature. Therefore, as in the Anosov case the curvature is not zero everywhere (cf. \cite{Ebe:73}), then  the Gauss-Bonnet Theorem  implies that if the geodesic flow is Anosov, then $\chi(M)<0$.\\
For non-compact surface of finite volume the Gauss-Bonnet theorem holds (cf. \cite{Ros:82}), then as a consequence of Corollary 
\ref{Cor1} we have
 \begin{C}\label{Cor6}
Any complete surface of finite volume, curvature bounded below and non-negative Euler characteristic does not admit a Riemannian metric whose  geodesic flow is Anosov.  
 \end{C}
In particular, if $S$ is a sphere, its Euler Characteristic equal to $2$, then no complete surface $M$ of finite volume
homeomorphic to $S$ with a point or two points deleted, has Anosov geodesic flow, since $\chi(M)=\chi(S)-1=1>0$ or $\chi(M)=\chi(S)-2=0$, see figure below.
\begin{figure}[htbp]
	\centering
	\subfloat[$\chi(M)=1$]{
		\includegraphics[height=4cm]{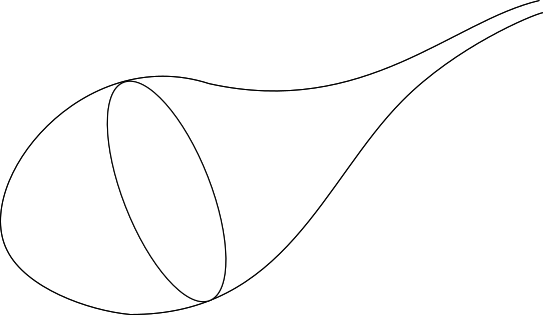}
		}
	\quad 
		\subfloat[$\chi(M)=0$]{
		\includegraphics[height=4cm]{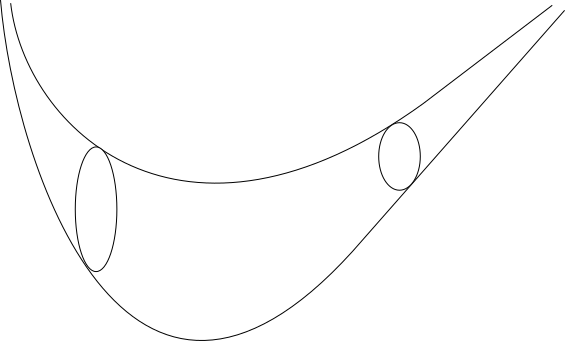}
		}
\end{figure}

\begin{proof}[\emph{\textbf{Proof of Corollary \ref{Cor2}}}]
Let us to  prove by contradiction. Assume that the geodesic flow is an Anosov flow, then by Theorem \ref{main} there are $B,t_0>0$ which satisfies (\ref{E1main}). Let $t>t_1$, where $t_1$ is as our hypotheses. Therefore,

\begin{equation}\label{e1Cor2}
\frac{1}{t}\int_{0}^{t}k(V(t))dt=\frac{1}{t}\int_{0}^{t_1}k(V(t))dt+\frac{1}{t}\int_{t_1}^{t}k(V(t))dt.
\end{equation}
Given any  $\delta<B$, we taken $t>\max\{t_0, t_1\}$ large enough such that\break $-\delta<\ds \frac{1}{t}\int_{0}^{t_1}k(V(t))dt<\delta$. Thus
by our hypotheses  $\ds \frac{1}{t}\int_{t_1}^{t}k(V(t))dt\geq 0$. Therefore, by equation (\ref{e1Cor2}) and the chosen of $t$, we have 
$-B>-\delta$, which is a contradiction. Thus, we concludes the proof of the first part of corollary. {To prove, the second part, note that 
in a manifold $M=N\times O$ endowed with the product metric, it is possible to construct a parallel perpendicular vector field along of a 
geodesic $\gamma$, totally contains in $N$ or $M$. Follows then, by the first part of corollary, that the geodesic flow of $M$ is not
Anosov.}
\end{proof}

We end this section with the proof of Corollary \ref{Cor3}.
\begin{proof}[\emph{\textbf{Proof of Corollary \ref{Cor3}}}]
It is known that non-compact manifold has rays, \emph{i.e.}, there exists $\theta=(p,v) \in SM$ such that the geodesic  
$\gamma_{\theta}:[0, \infty) \to M$ satisfies $d(\gamma_{\theta}(t),\gamma_{\theta}(s))=|s-t|$. Since $\mathcal{K}(t) \to 0$ as $t \to \infty$ , then $Ric_{\gamma_{\theta}(t)} (\gamma'_{\theta}(t))\to 0$ as $t \to \infty$. In particular, 
$$ \lim_{t\to +\infty} \Big| \displaystyle\frac{1}{t}\displaystyle\int_{0}^{t} Ric_{\gamma_{\theta}(r)} (\gamma'_{\theta}(r))\ dr \Big|=0.$$
Therefore, by Corollary \ref{Cor4},  follows that the geodesic flow of $M$ is not Anosov.
\end{proof}

At the end of Section \ref{Sec_Examples}, we construct a non-compact surface of negative curvature whose geodesic flow is not Anosov, showing that, from the point view of the dynamic, manifold of negative curvature and manifold of pinched negative curvature are totally different.   

\section{Proof of Theorem \ref{main2} }\label{Nec_Cond}
The main goal of this section is to prove  Theorem \ref{main2}.  The idea of the proof is to use the hypothesis of Theorem \ref{main2} to
show that  the stable  Jacobi field  has norm exponential decreasing for the future, and unstable Jacobi field  has norm exponential
increasing for the future (see Proposition \ref{P1main2}). Then, we use a strategy similar to Eberlein at \cite{Ebe:73}, to show the
uniform contraction of the stable and unstable bundle. \\

\begin{T}\label{main2*}
Let $M$ be a complete surface with curvature bounded below by $-c^2$ without focal points. Assume that, there are two constants $B, t_0>0$ such that for all geodesic $\gamma(t)$
\begin{equation}\label{e1main2}
\frac{1}{t}\int_{0}^{t} k(\gamma(s)) \ dr   \leq -B \ \ \text{whenever} \ \ t > t_0,
\end{equation} 
then the geodesic flow is an Anosov flow.
\end{T}


The following proposition gives us control of the determinant of the matrix of stable and unstable Jacobi fields. 

\begin{Pro}\label{P1main2}
Let $M$ be a complete manifold with curvature bounded below by $-c^2$ without focal point. Assume that, there are two constants $B, t_0>0$ such that for all geodesic $\gamma(t)$ and for each unit perpendicular parallel vector field $V(t)$ on $\gamma(t)$  we  have 
\begin{equation}\label{e2main2}
\frac{1}{t}\int_{0}^{t} k(V(r)) \ dr   \leq -B \ \ \text{whenever} \ \ t>t_0.
\end{equation} 
Then there are two positive constants $E$ and $t_1$ such that 
\begin{itemize}
\item[$(a)$] $|\operatorname{det}\, Y_{\theta, u}(t) |\geq e^{Et} \ \ \text{for all} \ \ t>t_1$,
\item[$(b)$] $|\operatorname{det}\, Y_{\theta, s}(t) |\leq e^{-Et} \ \ \text{for all} \ \ t>t_1.$
\end{itemize}
\end{Pro}
\begin{proof}
We prove item (b). The proof of item (a) is analogue to item (b). Indeed, as $Y_{\theta,s}(t)$ is the stable tensor given by (\ref{E1Jac}) 
and  satisfies the Jacobi equation (\ref{Jacobi}). 
Thus, $U_{\theta, s}(t):=U_{s}(\phi^{t}(\theta))=Y'_{\theta,s}(t)\cdot Y^{-1}_{\theta,s}(t)$ is solution of the Ricatti equation 
(\ref{Ricatti}) (see Subsection \ref{NCP}). As $M$ has no focal points, then for each $x\in \re^n\setminus \{0\}$, we have that the function
$t\to |Y_{\theta,s}(t)x|^2$ is decreasing (cf. \cite{Ebe:73}), \emph{i.e.},
$$\frac{d}{dt}|Y_{\theta,s}(t)x|^2=2\langle Y_{\theta,s}(t)x , Y'_{\theta,s}(t)x \rangle\leq 0, \ \ \, x\in \re^{n-1}.$$
Therefore, as $Y_{\theta,s}(t)$ is invertible we have $$\langle y, U_{\theta,s}(t)y\rangle=\langle y, Y'_{\theta,s}(t)\cdot Y^{-1}_{\theta,s}(t)y\rangle\leq 0, \ \ \ y\in \re^{n-1}.$$ 
Since $U_{\theta,s}(t)$ is symmetric, the last equation implies that all eigenvalues of $U_{\theta,s}(t)$ are non-positive. 
Let $-\lambda_{n-1}(t)\leq -\lambda_{n-2}(t)\leq \cdots \leq -\lambda_{1}(t)\leq 0$ the eigenvalues of $U_{\theta,s}(t)$, then  as $|U_{\theta,s}(t)|\leq c$, then $0\leq \lambda_i(t)\leq c$, $i=1,2,\dots, n-1$ which provides 
\begin{eqnarray}\label{e2'main2}
\text{tr}\, (U_{\theta,s}(t))^{2}&=&\lambda^2_1(t)+\lambda^2_2(t)+\cdots +\lambda^2_{n-1}(t)\nonumber \\ 
&\leq & c(\lambda_1(t)+\lambda_2(t)+\cdots +\lambda_{n-1}(t))\nonumber\\  &=&-c \, \text{tr}\,U_{\theta,s}(t).
\end{eqnarray}
Taking trace in the equation (\ref{Ricatti}) and integrating we have 
\begin{eqnarray}\label{e3main2}
0 &=&\frac{1}{t}\int_{0}^{t}\text{tr} \,U'_{\theta,s}(r)dr + \frac{1}{t}\int_{0}^{t}\text{tr}\,U^2_{\theta,s}(r)dr+\frac{1}{t}\int_{0}^{t}\text{tr}\, R(r)dr \nonumber \\
&=&\frac{\text{tr} \, U_{\theta,s}(t)-\text{tr} \, U_{\theta,s}(0)}{t}+\frac{1}{t}\int_{0}^{t}\text{tr}\,U^2_{\theta,s}(r)dr+\frac{1}{t}\int_{0}^{t}\text{tr}\,R(r)dr. 
\end{eqnarray}
For $t>t_0$,  our hypothesis (equation (\ref{e2main2})) implies that $\ds\frac{1}{t}\int_{0}^{t}\text{tr}\,R(r)dr<-(n-1)\cdot B$, as $|\text{tr}\,U_{\theta,s}(r)|\leq (n-1)c$. {From equation}  (\ref{e3main2}) there is $t_1 > 0$ such that 
\begin{equation}\label{e4main2}
\ds\frac{1}{t}\int_{0}^{t}\text{tr}\,U^2_{\theta,s}(r)dr>\frac{(n-1)\cdot B }{2}, \ \  t>t_1.
\end{equation}
The equations (\ref{e2'main2}) and (\ref{e4main2}) provides 
\begin{equation}\label{e5main2}
\int_{0}^{t}\text{tr}\,U_{\theta,s}(r)dr\leq -\frac{(n-1)\cdot B}{2c}\, t:=-Et, \, \,\, \, t>t_1.
\end{equation}
To conclude the argument, we remember the Liouville's Formula or {Jacobi's Formula}, (see \cite[Lemma 4.6]{MelRom:pre} or \cite{ManFre:82}) which states that
\begin{equation}\label{e5'main2}
\frac{d}{dt}\log | \text{det}\, Y_{\theta,s}(r)|=\text{tr}\, U_{\theta,s}(r), \ \ Y_{\theta,s}(0)=I.
\end{equation}
 \noindent Integrating this last equation and using (\ref{e5main2}) we obtain 
\begin{equation}\label{e6main2}
| \text{det}\, Y_{\theta,s}(t)|\leq e^{-Et}, \ \ \ t>t_1,
\end{equation}
which concludes the proof of proposition.
\end{proof} 
\begin{C}\label{C1main2}
In the same conditions of \emph{Proposition \ref{P1main2}}, there exists $t_2 > 0$ such that for each $\theta\in SM$ we have $$\frac{1}{t}\log |\text{det}\, D\phi^t|_{E^s(\theta)}|\leq -\frac{E}{2} \ \ \text{and} \ \ \frac{1}{t}\log |\text{det}\, D\phi^t|_{E^u(\theta)}|\geq \frac{E}{2}, \ \ \text{whenever} \ \ {t>t_2}.$$
\end{C}
\begin{proof}
Following the same lines of Lemma 4.6 from {\cite{MelRom:pre}}, consider  for each $\theta\in SM$  the subspace $N(\theta)$ of $T_pM$
orthogonal to $v$. Then\\
${E^{s(u)}({\theta})}=\text{graph} \,  U_{\theta,s(u)}=\{(x,U_{\theta,s(u)}x):x\in N(\theta)\}$ and
\begin{equation}\label{e7main2}
d\phi^{t}|_{E^{s(u)}({\theta})}=\pi^{-1}_{\phi^t(\theta),s(u)}\circ Y_{\theta,s(u)}(t)\circ \pi_{\theta,s(u)},
\end{equation} 
where $\pi_{\theta,s(u)}\colon E^{s(u)}(\theta)\to N(\theta)$ is the projection in the first coordinate. 
This projection satisfies  (see \cite[equation 4.12]{MelRom:pre})
\begin{equation}\label{e8main2}
1\leq |\text{det}\,\pi^{-1}_{\phi^t(\theta)} |\leq (1+c^2)^\frac{n-2}{2}.
\end{equation}
The equations (\ref{e7main2}) and (\ref{e8main2}) and Proposition \ref{P1main2} provides that there exists $t_2 >0$ such that\\

\noindent $\bullet$  \textbf{Stable case:}
\begin{eqnarray}\label{e9'main2}
\frac{1}{t}\log |\text{det}\, D\phi^t|_{E^s(\theta)}|&=&\frac{1}{t}\log |\text{det}\pi^{-1}_{\phi^{t}(\theta),s}|+\frac{1}{t}\log | \text{det}\, Y_{\theta,s}(t)|+\frac{1}{t}\log |\text{det}\,\pi_{\theta}| \\ \label{e9main2}
&\leq & \frac{(1+c^2)^\frac{n-2}{2}}{t}-E\leq -\frac{E}{2} \,\,\, \ \ \text{for all} \ \ t>t_2,
\end{eqnarray}
since $|\text{det}\,\pi_{\theta}|\leq||\pi_{\theta}||^{n-1}\leq 1$.\\

\noindent $\bullet$ \textbf{Unstable case:}
\begin{eqnarray}\label{e10'main2}
\frac{1}{t}\log |\text{det}\, D\phi^t|_{E^u(\theta)}|&=&\frac{1}{t}\log |\text{det}\pi^{-1}_{\phi^{t}(\theta),u}|+\frac{1}{t}\log | \text{det}\, Y_{\theta,u}(t)|+\frac{1}{t}\log |\text{det}\,\pi_{\theta}|  \\ \label{e10main2}
&\geq & E+ \frac{(1+c^2)^{-\frac{n-2}{2}}}{t}\geq \frac{E}{2} \,\,\, \ \ \text{for all} \ \ t>t_2,
\end{eqnarray}
since $|\text{det}\,\pi_{\theta}|=\frac{1}{|\text{det}\,\pi^{-1}_{\theta}|}\geq (1+c^2)^{-\frac{n-2}{2}}$. 
Thus, we conclude the proof of corollary.
\end{proof}
\begin{R}
It is worth noting that \emph{Proposition \ref{P1main2}} and \emph{Corollary \ref{C1main2}} holds in any dimension. Therefore, we believe
that these can be used to proof the \emph{Theorem \ref{main2}} in any dimension, in other words, it is still needed to be explored if it can help prove
the conjecture given in the introduction for higher dimensions. 
\end{R}

To make the proof of Theorem \ref{main2}, let us use the Corollary \ref{C1main2} and the following lemma, which was proven in a similar
version by Eberlein at \cite[Lemma 3.12]{Ebe:73}, but we present a proof with weaker hypotheses (see condition (ii) at Lemma \ref{L1main2}).
\begin{Le}\label{L1main2}
Let $f\colon (0,+\infty)\to (0,+\infty)$ a bounded function such that 
\begin{enumerate}
\item[\emph{(i)}] $f(t+s)\leq f(t)f(s)$ for every $s,t$;
\item[\emph{(ii)}] There is $r>0$ such that $f(r)<1$. 
\end{enumerate}
Then, there are constants $C>0$, $\lambda\in (0,1)$ such that $$f(t)\leq C\lambda^t, \ \ \ t>0.$$  
\end{Le}
\begin{proof}
Let $k\in \mathbb{N}$, then $f(kr)\leq f(r)^{k}$. Therefore, for each $t>0$, we can write $t=kr+m$ for $k\in \mathbb{N}$, $m\in (0,r)$ and 
we have then 
$$f(t)\leq f(kr)f(m)\leq f(r)^kf(m).$$
Let $F$ be the upper bound of $f$, and let $b=f(r)<1$. Then, we have 
$$f(t)\leq F\cdot b^k=F\cdot b^{-\frac{m}{r}}(b^{\frac{1}{r}})^t\leq F\cdot b^{-1}(b^\frac{1}{r})^{t}.$$
Thus, we take $C=F\cdot b^{-1}$ and $\lambda=b^\frac{1}{r}$.
\end{proof}
We are ready to prove of Theorem \ref{main2}
\begin{proof}[\emph{\textbf{Proof of Theorem \ref{main2}}}]
As $M$ is a surface, then $\left|\text{det}\, D\phi^t|_{E^{s(u)}(\theta)}\right|=\left\|D\phi^t|_{E^{s(u)}(\theta)}\right\|$. By 
Corollary \ref{C1main2} for all $\theta\in SM$ and $t>t_2$
\begin{equation}\label{e12main2}
\left\|D\phi^t|_{E^{s}(\theta)}\right\|\leq e^{-\frac{E}{2}t} \ \ \text{and} \ \ \ \left\|D\phi^{-t}|_{E^{u}(\theta)}\right\|\leq e^{-\frac{E}{2}t}, \ \ \ t>t_2.
\end{equation}

\noindent \textbf{Claim:} 
There is $G>0$ such that
$$\left\|D\phi^t|_{E^{s}(\theta)}\right\|\leq e^G \ \ \text{and} \ \ \left\|D\phi^{-t}|_{E^{u}(\theta)}\right\|\leq e^G, \ \ \ t\in [0,t_2].$$
\begin{proof}[\emph{\textbf{Proof of Claim:}}]
It is sufficiently to  note that, when the manifold has no focal points, then the norm of stable Jacobi fields are decreasing and 
$\left\|U_{\theta,s}\right\|\leq c$. Analogue to unstable case. 

	


\end{proof}

\noindent To concludes the proof of Theorem \ref{main2}, we consider the following two functions:
$$f_{s}(t)=\sup_{\theta\in SM}\left\|D\phi^t|_{E^{s}(\theta)}\right\| \ \ \text{and} \ \ f_{u}(t)=\sup_{\theta\in SM}\left\|D\phi^t|_{E^{u}(\theta)}\right\|.$$
Both of the functions are bounded by the above claim and inequality (\ref{e12main2}). These function are also sub-additive, because each $D\phi^t|_{E^{s(u)}(\theta)}$ are linear operators, \emph{i.e.}, satisfies the item (i) of the Lemma \ref{L1main2}. The inequalities at (\ref{e12main2}) also show that there is $r>0$ such that for all $\theta\in SM$ 
$$\left\|D\phi^r|_{E^{s}(\theta)}\right\|<1 \ \ \text{and} \ \ \left\|D\phi^{-r}|_{E^{u}(\theta)}\right\|<1,$$
which implies that $f_{s(u)}(t)$ satisfies the item (ii) of Lemma \ref{L1main2}. Thus, by Lemma \ref{L1main2} there are $C_{s(u)}>0$, $\lambda_{s(u)}\in (0,1)$ such that $$f_{s(u)}(t)\leq C_{s(u)}\lambda_{s(u)}^t.$$
We taken  $C=\max\{C_s, C_u\}$ and $\lambda=\max\{\lambda_s,\lambda_u\}$ to  conclude that, for $\theta\in SM$

$$\left\|D\phi^t|_{E^{s}(\theta)}\right\|\leq C\lambda^t \ \ \text{and} \ \ \left\|D\phi^{-t}|_{E^{u}(\theta)}\right\|\leq C\lambda^t, \ \ \ t\geq 0.$$
The last inequalities allows us to state that the subspaces  $E^{s(u)}(\theta)$ are linearly independent. Moreover, since $M$ has no focal points, then  $E^{s(u)}(\theta)$ are continuous in $\theta$ (cf. \cite{Ebe:73}). Thus, we  concludes that the geodesic flow is Anosov. 
\end{proof}
\section{Examples of Anosov Geodesic Flows on Non-compact Surfaces}\label{Sec_Examples}
In this section, we use the Theorem \ref{main2} to construct a family of non-compact (non-compactly homogeneous) surfaces, all diffeomorphic to the cylinder $\re \times \mathbb{S}^1$, whose geodesic flow is Anosov (see Section \ref{Examples}). A key tool to build this family of surfaces is the ``\emph{Warped product}", which will be described in the following subsection. 
\subsection{Warped Products}
Let $M,N$ be Riemannian manifolds, with metrics $g_M$ and $g_N$, respectively. Let  $f  >  0$ be a smooth function on $M$. The  \emph{warped product} of $M$ and $N$, $S  =   M \times_{f}  N$, 
is the product manifold  $M  \times  N$  furnished with the Riemannian  metric 
$$    g = \pi_M^{*}( g_M) + (f \circ \pi_M)^2  \pi_N^{*} (g_N),                             $$
where    $\pi_M$  and  $\pi_N$  are the projections  of  $M \times  N$  onto $M$ and  $N$, respectively.

Let $X$ be  a vector field on $M$. The  horizontal lift of $X$ to $M\times_f N$ is the vector field $\overline{X}$ such that $d{\pi_M}_{(p,q)}(\overline{X}(p,q)) = X(p)$ and $d{\pi_N}_{(p,q)}(\overline{X}(p,q)) = 0$. If $Y$ is a vector field on $N$, the vertical lift of $Y$ to $M\times_{f} N$ is the vector field $\overline{Y}$ such that $d{\pi_M}_{(p,q)}(\overline{Y}(p,q)) = 0$ and $d{\pi_N}_{(p,q)}(\overline{Y}(p,q)) = Y(q)$. The set of all such lifts are denoted, as usual, by $\mathcal{L}(M)$ and $\mathcal{L}(N)$, respectively.

We denote by  $\nabla$, $\nabla^M$ and $\nabla^{N}$ the Levi-Civita connections on $M \times_f N$, $M$ and $N$, respectively.\\
The following proposition describes the relationship between the above connections.

\begin{Pro}\label{conexao}\emph{{\cite {Neil:83}}}
	On  $S  =  M  \times_f  N$,  if  $\overline{X}, \overline{Y} \in \mathcal{L}(M)$ and  $\overline{U}, \overline{V} \in \mathcal{L}(N)$ then
	\begin{enumerate}
		\item[\emph{1.}]  $\nabla_{\overline{X}}\overline{Y} = \overline{\nabla_X^M Y},   $\\
		\item[\emph{2.}]  $\nabla_{\overline{U}}\overline{X} = \nabla_{\overline{X}}\overline{U}  =   \left( {\dfrac{Xf}{f}}\right) \overline{U},$\\
		\item[\emph{3.}] $d\pi_M( \nabla_{\overline{U}}\overline{V})= -(g(U,V)/f) \cdot {\rm grad}\, f$,\\
		\item[\emph{4.}] $d\pi_{N}( \nabla_{\overline{U}}\overline{V} ) = \nabla_{U}^{N} V.$
		\item[\emph{5.}] If $\overline{X}$ and $\overline{U}$ are unit vectors then $K(\overline{X}, \overline{U}) = -(1/f){\rm{Hess}}_{M}f(X, X)$, where $K$ denotes the sectional curvature of the plane spanned by $\overline{X}$ and $\overline{U}$.
	\end{enumerate}
\end{Pro}
\subsection{Family of Non-Compact Surfaces with Anosov Geodesic Flow}\label{Examples}
Finally, in this section, let us to construct a surface (diffeomorphic to cylinder) of non-positive curvature with  Anosov geodesic flow (see Subsection \ref{example1}). Furthermore, using a similar arguments of Example 1, we construct, using the Warped Product and the Corollary \ref{Cor3}, a surface of negative curvature whose geodesic flow is not Anosov (see Subsection \ref{example2}).
\subsection{Example 1}\label{example1}
Consider the warped product $M = \mathbb{R} \times_{f} \mathbb{S}^1$, where $f(x) = e^{g(x)}$ and $g(x)$ is a smooth function such that
\begin{itemize}
	\item[{(A)}] $g''(x) + (g'(x))^2 \geq 0$, for any $x$;
	\item[{(B)}] $g''+ (g')^2$ is a periodic function with period $T > 0$;
	\item[{(C)}] There are positive constants $C_1$ and $C_2$ such that $ C_1/2 < g' < C_2/2 $. 
\end{itemize} 

Find functions $g$ that satisfies the above three conditions is very easy, for example, consider the one-parameter family  of function 
$g(x) = ax - \cos x + \sin x$, for $a>0$ {large enough}.

Observe that from Proposition \ref{conexao} and condition (A), the curvature in the point $(x,y)$ of the surface $M$ is given by 
\begin{equation}\label{e0exam1}
K(x,y) = K(x)= -f''/f = -(g''(x) + (g'(x))^2) \leq 0.
\end{equation}
 In particular, from the condition (B), the function $K$ is periodic with period $T$ and $M$ has no focal points, since the curvature is non-positive.  
Throughout the remainder of this section we show that $M$ satisfies the equation (\ref{e1main2}). In this direction, let us understand the geodesics in $M$, looking at the  local coordinates.

Consider a geodesic $\gamma(t) = (x(t), z(t))$ in $M$ with $|\gamma'(t)| = 1$ and a parametrization 
 $\varphi_{t_0}: \mathbb{R} \times (t_0, t_0 + 2 \pi) \to M$
  where $\varphi_{t_0}(x,y) = (x, \cos y, \sin y)$ and $$\varphi_{t_0}(\mathbb{R} \times (t_0, t_0 + 2 \pi) ) \cap \gamma(\mathbb{R}) \neq \varnothing.$$ 
  Let $(x(t), y(t))$ be the local expression of $\gamma(t)$. Consider $X_1 = \displaystyle\frac{\partial}{\partial x}$ and $X_2 = \displaystyle\frac{\partial}{\partial y}$, by the Proposition \ref{conexao} we have
$$  \nabla_ {X_1} X_1 = 0, \, \nabla_ {X_1} X_2 =  \nabla_ {X_2} X_1 = g' X_2, \, {\rm and} \, \nabla_ {X_2} X_2 = - e^{2g}g'X_1.                                                 $$
The Christoffel symbols are given by
$$ \Gamma_{11}^1 = \Gamma_{11}^2 = \Gamma_{12}^1 = \Gamma_{21}^1 = \Gamma_{22}^2 =0, \, \Gamma_{12}^2 = \Gamma_{21}^2 = g' \, {\rm and} \, \Gamma_{22}^1 = -e^{2g}g'.$$

Since $\gamma$ is a geodesic with  $|\gamma'(t)| = 1$ the functions $x(t)$ and $y(t)$ satisfy the following equalities,
\begin{itemize}
\item $x''(t) - e^{2g(x(t))}g'(x(t))(y'(t))^2 = 0$,
\item $y''(t) + 2g'(x(t))x'(t)y'(t) = 0$,
\item $(x'(t))^2 + e^{2g(x(t))}(y'(t))^2 = 1$.
\end{itemize}
Observe that 
\begin{equation}\label{geodesic}
x''(t) = g'(x(t))(1 - (x'(t))^2)
\end{equation} 
 and $|x'(t)| \leq 1$ for every $t \in \mathbb{R}$,  since  this equality does not depend of the parametrization.
 
If there exists $a \in \mathbb{R}$ such that $|x'(a)| = 1$, then $z(a) = 0$. It follows, by the uniqueness of the geodesics,  that $\gamma(t) = (x(t), z(t)) = ( x(a) + t- a, y(a))$ or $\gamma(t) = (x(t)  , z(t)) = ( x(0)+ a - t, y(0))$ .

Now assume that  $|x'(t)| <1$ for every $t \in \mathbb{R}$. Set $b(t) = x'(t)$, from (\ref{geodesic}) we have
\begin{equation}\label{principal}
       \displaystyle\frac{b'(t)}{1 - (b(t))^2}           = g'(x(t)).                                   
\end{equation}
Thus, 
$$   \displaystyle\frac{1}{2}\Big( \log \Big(\displaystyle\frac{1 + b(t)}{1 - b(t)}  \Big)            \Big)'         = g'(x(t)).                                   $$                                 Integrating, we get
$$        \log \Big(\displaystyle\frac{1 + b(t)}{1 - b(t)}  \Big) -  \log \Big(\displaystyle\frac{1 + b(0)}{1 - b(0)}  \Big)    = 2 \displaystyle\int_{0}^{t} g'(x(s))  \, ds.                                          $$        
Hence, the condition (C) for $g$ provides 
$$ B_0 e^{C_{1}t} <   \displaystyle\frac{1+ b(t)}{1 - b(t)} <  B_0e^{C_{2}t}, $$
where $B_0 = \frac{1+b(0)}{1-b(0)}>0$, since $|x'(t)|<1$.
This implies
\begin{equation}\label{des}
1 - \displaystyle\frac{2}{B_0e^{C_{1}t}+1} < b(t) < 1 - \displaystyle\frac{2}{B_0e^{C_{2}t}+1}.
\end{equation}

\noindent Using the above 
equalities and inequalities we will study the below expression
$$         \displaystyle\frac{1}{t}\displaystyle\int_{0}^{t} K(\gamma(s)) \ ds   =  \displaystyle\frac{1}{t}\displaystyle\int_{0}^{t} K(x(s)) \ ds.$$

\noindent In this direction, we will divide the analysis in some cases, regarding the position  of $b(0)$ in $[-1,1]$.
\ \\
\noindent \textbf{Case 1:}\, \, $b(0) = x'(0) = 1$.
	
	In this case $x(t) = x(0) + t$. Hence,
	\begin{eqnarray*}
		\displaystyle\frac{1}{t}\displaystyle\int_{0}^{t} K(x(s)) \ ds &=& \displaystyle\frac{1}{t}\displaystyle\int_{0}^{t} K( x(0) + s) \ ds\\
		&=&  \displaystyle\frac{1}{t}\displaystyle\int_{x(0)}^{x(0) +t} K( u) \ du.
	\end{eqnarray*} 
	Take $ t > 2T$, where $T$ is the period of the function $g$. We can to write $t = n_t T + a$ where $n_t$ is a positive integer number and $0 \leq a < T$. Set $\eta := \displaystyle\int_{0}^{T} K(s) ds <0$, observe that
	\begin{eqnarray*}
		\displaystyle\frac{1}{t}\displaystyle\int_{0}^{t} K(x(s)) \ ds &=& \displaystyle\frac{1}{t}\displaystyle\int_{x(0)}^{x(0) +t} K( u) \ du\\
		&=& \frac{1}{t}\displaystyle\sum_{i=1}^{n_t} \displaystyle\int_{x(0) + (i-1)T}^{x(0) + iT} K(u) du + \frac{1}{t}\displaystyle\int_{x(0) + n_tT}^{x(0) + t} K(u) du \\
		&\leq& \displaystyle\frac{\eta n_t}{t}\\
		&=& \displaystyle\frac{\eta}{T} - \displaystyle\frac{a\eta}{tT} \\
		&\leq& \displaystyle\frac{\eta}{2T}.
	\end{eqnarray*}   

\noindent \textbf{Case 2:}\, \, $b(0) = x'(0) = -1$.\\
Proceeding in the same way, we get      
$$     \displaystyle\frac{1}{t}\displaystyle\int_{0}^{t} K(x(s)) \ ds \leq  \displaystyle\frac{\eta}{2T},                                              $$       
for $t > 2T$.\\
\ \\
\noindent \textbf{Case 3:}\, \, $1/2 \leq b(0) = x'(0) < 1$.\\
Observe that by (\ref{principal}) and (\ref{des}), $b'(t) > 0$ and $b(t) < 1$. In particular, $b(t)$ is a strictly increasing function and $1/2 < b(t) < 1$ for every $t > 0$. Hence, $x(t)$ is an increasing function.\\
Consider the change of variable $u = x(s)$. We have
\begin{eqnarray*}
	\displaystyle\frac{1}{t}\displaystyle\int_{0}^{t} K(x(s)) \ ds &=& \displaystyle\frac{1}{t}\displaystyle\int_{x(0)}^{x(t)} \displaystyle\frac{K(u)}{x'(x^{-1}(u))} \ du\\
	& \leq& \displaystyle\frac{1}{t}\displaystyle\int_{x(0)}^{x(t)} K(u) \ du.\\
\end{eqnarray*}
Take $ t > 4T$ and write $t/2 = n_t T + a$ where $n_t$ is a positive integer number and $0 \leq a < T$. Since $x'(t) > 1/2$ for $t > 0$, it follows that $x(t) > x(0) + 1/2 \cdot t$ for $t >0$. Hence,

\begin{eqnarray*}
	\displaystyle\frac{1}{t}\displaystyle\int_{0}^{t} K(x(s)) \ ds &\leq& 
	\displaystyle\frac{1}{t}\displaystyle\int_{x(0)}^{x(t)} K(u) \ du\\
	&=& \displaystyle\frac{1}{t}\displaystyle\int_{x(0)}^{x(0) + t/2} K(u) \ du + \displaystyle\frac{1}{t}\displaystyle\int_{x(0) + t/2}^{x(t)} K(u) \ du \\
	&\leq&  \displaystyle\frac{1}{t}\displaystyle\int_{x(0)}^{x(0) + t/2} K(u) \ du \\
	&=&  \displaystyle\frac{1}{t}\displaystyle\int_{x(0)}^{x(0) + n_tT} K(u) \ du + \displaystyle\frac{1}{t}\displaystyle\int_{x(0)+ n_tT}^{x(0) + t/2} K(u) \ du \\
	& \leq & \displaystyle\frac{1}{t}\displaystyle\int_{x(0)}^{x(0) + n_tT} K(u) \ du. \\
\end{eqnarray*}
Since $K$ is a periodic function with period $T$, it follows that
\begin{eqnarray*}
	\displaystyle\frac{1}{t}\displaystyle\int_{0}^{t} K(x(s)) \ ds &\leq&  \displaystyle\frac{1}{t} \displaystyle\sum_{i=1}^{n_t} \displaystyle\int_{x(0) + (i-1)T}^{x(0) + iT} K(u) \ du\\
	&=& \displaystyle\frac{n_t\eta}{t}\\
	&=& \displaystyle\frac{\eta}{2T} - \displaystyle\frac{a\eta}{tT}\\
	& \leq& \displaystyle\frac{\eta}{4T}.
\end{eqnarray*}

\noindent \textbf{Case 4:}\, \, $-1/2 \leq b(0) = x'(0) < 1/2$.\\
Now consider $-1/2 \leq b(0) = x'(0) < 1/2$. By (\ref{des}), $\displaystyle\lim_{t \to + \infty} b(t) = \displaystyle\lim_{t \to + \infty} x'(t) = 1$, note that there is a unique $T_1 > 0$ such that $b(T_1) = 1/2$ because the function $b(t)$ is strictly increasing. Hence, by (\ref{des})
 
 $$       1 - \displaystyle\frac{2}{B_0e^{C_{1}T_1}+1}  < \displaystyle\frac{1}{2},                                   $$
\noindent which implies, $T_1 < \textcolor{red}\ds\dfrac{1}{C_1} \log \Big(\dfrac{3}{B_0}\Big) < \ds\frac{2}{C_1} \log 3 $, since $B_0>\dfrac{1}{3}$. \\
Take $ t > \max \Big \{\displaystyle\frac{2}{C_1} \log 3 + 4T, \displaystyle\frac{4}{C_1} \log 3 \Big \}$, we have
 
 $$    \displaystyle\frac{1}{t}\displaystyle\int_{0}^{t} K(x(s)) \ ds = \displaystyle\frac{1}{t}\displaystyle\int_{0}^{T_1} K(x(s)) \ ds + \displaystyle\frac{1}{t}\displaystyle\int_{T_1}^{t} K(x(s)) \ ds    \leq   \displaystyle\frac{1}{t}\displaystyle\int_{T_1}^{t} K(x(s)) \ ds.                                       $$
{Now consider the geodesic  $\beta (t) = \gamma(t+T_1)$ and apply the inequality in the Case 3, we have}
 $$\displaystyle\int_{T_1}^{t} K(x(s +T_1)) \ ds  =  \displaystyle\int_{0}^{t-T_1} K(x(s)) \ ds   \leq \displaystyle\frac{\eta}{4T}(t - T_1).                                          $$
 Hence,
 \begin{eqnarray*}
 	\displaystyle\frac{1}{t}\displaystyle\int_{0}^{t} K(x(s)) \ ds & \leq& \displaystyle\frac{1}{t}\displaystyle\int_{T_1}^{t} K(x(s)) \ ds\\
 	& \leq& \displaystyle\frac{\eta}{4T}\Big(1 - \displaystyle\frac{T_1}{t} \Big) \\
 	& \leq & \displaystyle\frac{\eta}{8T}.
 \end{eqnarray*}
\textbf{Case 5:}  $-1 < b(0) = x'(0) < -1/2$.\\
 By (\ref{des}), $\displaystyle\lim_{t \to + \infty} b(t) = \displaystyle\lim_{t \to + \infty} x'(t) = 1$, note that there is a unique $T_2 > 0$ such that $b(T_2) = -1/2$ because the function $b(t)$ is strictly increasing. Using again the inequality (\ref{des})

$$     -\displaystyle\frac{1}{2} < 1 -  \displaystyle\frac{2}{B_0e^{C_{2}T_2}+ 1}                                    $$
which implies, 
\begin{equation}\label{ine}
T_2 > \displaystyle\dfrac{1}{C_2} \log \Big(\displaystyle\frac{1}{3B_0} \Big).
\end{equation}
Note that $B_0 \to 0$ when ${b(0)} \to -1$. {In particular, $T_2 \to + \infty$ as $b(0) \to -1$ }. So, let us start first suppose that $T_2 \leq 4T$. In this case, take  $ t > 4T + \max \Big\{\displaystyle\frac{2}{C_1} \log 3 + 4T, \displaystyle\frac{4}{C_1} \log 3 \Big\}$. We have
  $$    \displaystyle\frac{1}{t}\displaystyle\int_{0}^{t} K(x(s)) \ ds = \displaystyle\frac{1}{t}\displaystyle\int_{0}^{T_2} K(x(s)) \ ds + \displaystyle\frac{1}{t}\displaystyle\int_{T_2}^{t} K(x(s)) \ ds    \leq   \displaystyle\frac{1}{t}\displaystyle\int_{T_2}^{t} K(x(s)) \ ds.                                       $$
   Observe that $t - T_2 > \max \Big\{\displaystyle\frac{2}{C_1} \log 3 + 4T, \displaystyle\frac{4}{C_1} \log 3 \Big\}$. {Now consider the 
   geodesic\break  $\beta (t) = \gamma(t+T_2)$ and apply the inequality in the Case 4, we have}
  
 $$  \displaystyle\int_{T_2}^{t} K(x(s)) \ ds   \leq \displaystyle\frac{\eta}{8T}(t - T_2).                                          $$
Hence,
\begin{eqnarray*}
	\displaystyle\frac{1}{t}\displaystyle\int_{0}^{t} K(x(s)) \ ds & \leq& \displaystyle\frac{1}{t}\displaystyle\int_{T_2}^{t} K(x(s)) \ ds\\
	& \leq& \displaystyle\frac{\eta}{8T}\Big(1 - \displaystyle\frac{T_2}{t} \Big).
\end{eqnarray*}  
  Observe that $t > {8}T \geq 2T_2$ thus $T_2/t < 1/2$. Therefore,
  $$    \displaystyle\frac{1}{t}\displaystyle\int_{0}^{t} K(x(s)) \ ds < \displaystyle\frac{\eta}{16T}.$$                                           
  
  Now suppose that $T_2 > 4T$. For $4T < t \leq T_2$ we have that $-1 < x'(t)  \leq -1/2$. In particular, $x(t) \leq x(0) - 1/2 \cdot t$ for $4T < t \leq T_2$. Hence,
 \begin{eqnarray*}
	\displaystyle\frac{1}{t}\displaystyle\int_{0}^{t} K(x(s)) \ ds &=& \displaystyle\frac{1}{t}\displaystyle\int_{x(0)}^{x(t)} \displaystyle\frac{K(u)}{x'(x^{-1}(u))} \ du\\
	&=& \displaystyle\frac{1}{t}\displaystyle\int_{x(t)}^{x(0)} -\displaystyle\frac{K(u)}{x'(x^{-1}(u))} \ du\\
	& \leq& \displaystyle\frac{1}{t}\displaystyle\int_{x(t)}^{x(0)} K(u) \ du.\\
	&=& \displaystyle\frac{1}{t}\displaystyle\int_{x(t)}^{x(0) -t/2} K(u) \ du  + \displaystyle\frac{1}{t}\displaystyle\int_{x(0) -t/2}^{x(0)} K(u) \ du.\\
	& \leq& \displaystyle\frac{1}{t}\displaystyle\int_{x(0) -t/2}^{x(0)} K(u) \ du.	
\end{eqnarray*}  
   We can to write $t/2 = n_t T + a$ where $n_t$ is a positive integer number and $0 \leq a < T$. Hence,
 \begin{eqnarray*}
 \displaystyle\frac{1}{t}\displaystyle\int_{0}^{t} K(x(s)) \ ds &\leq& \displaystyle\frac{1}{t}\displaystyle\int_{x(0) -t/2}^{x(0)} K(u) \ du\\
 & \leq&
  \displaystyle\frac{1}{t} \displaystyle\sum_{i=1}^{n_t} \displaystyle\int_{x(0) -i T}^{x(0) - (i-1)T} K(u) \ du + \displaystyle\frac{1}{t}\displaystyle\int_{x(0) -t/2}^{x(0) - n_tT} K(u) \ du\\	
  & \leq & \displaystyle\frac{1}{t} \displaystyle\sum_{i=1}^{n_t} \displaystyle\int_{x(0) -i T}^{x(0) - (i-1)T} K(u) \ du\\
  &=& \displaystyle\frac{\eta n_t}{t}\\
  &=& \displaystyle\frac{\eta}{2T} - \displaystyle\frac{a\eta}{tT}\\
  & <& \displaystyle\frac{\eta}{4T},
 \end{eqnarray*} 
because $ t > 4T$. If $T_2 < t \leq T_2 +  \max \Big\{\displaystyle\frac{2}{C_1} \log 3 + 4T, \displaystyle\frac{4}{C_1} \log 3 \Big\}$ we have
$$ \displaystyle\frac{1}{t}\displaystyle\int_{0}^{t} K(x(s)) \ ds \leq  \displaystyle\frac{1}{t}\displaystyle\int_{0}^{T_2} K(x(s)) \ ds < \displaystyle\frac{\eta T_2}{4Tt}.$$
Set $A =\displaystyle\frac{2}{C_1} \log 3$. Since $T_2 > 4T$ we have

$$   \displaystyle\frac{T_2 + 2A + 4T}{T_2} < \displaystyle\frac{4T + 4T + 2A}{4T},                                                    $$
which provides 
$$   \displaystyle\frac{T_2}{t} \geq  \displaystyle\frac{T_2}{T_2 + \max  \Big\{\displaystyle\frac{2}{C_1} \log 3 + 4T, \displaystyle\frac{4}{C_1} \log 3  \Big \}}  \geq \displaystyle\frac{T_2}{T_2 + 2A + 4T} > \displaystyle\frac{4T}{8T+2A}.                                   $$
Therefore,

$$    \displaystyle\frac{1}{t}\displaystyle\int_{0}^{t} K(x(s)) \ ds < \displaystyle\frac{\eta}{8T + 2A}.                                 $$
If $t > T_2 +  \max \Big\{\displaystyle\frac{2}{C_1} \log 3 + 4T, \displaystyle\frac{4}{C_1} \log 3 \Big\}$,  {consider the geodesic  $\beta (t) = \gamma(t+T_2)$ and apply again  the inequality in the Case 4, we have} 
\begin{eqnarray*}
	\displaystyle\int_{0}^{t} K(x(s)) \ ds &=& 	\displaystyle\int_{0}^{T_2} K(x(s)) \ ds + 	\displaystyle\int_{T_2}^{t} K(x(s)) \ ds \\
	&<&  \displaystyle\frac{\eta T_2}{4T} + \displaystyle\frac{\eta(t-T_2)}{8T}\\
	&=&  \displaystyle\frac{\eta t}{8T} + \displaystyle\frac{\eta T_2}{8T}\\  
	&<& \displaystyle\frac{\eta t}{8T}.
\end{eqnarray*}
Thus,
$$        \displaystyle\frac{1}{t}\displaystyle\int_{0}^{t} K(x(s)) \ ds <      \displaystyle\frac{\eta}{8T}.                                       $$
Therefore, we prove that if $ t > 4T +  \max \Big\{\displaystyle\frac{2}{C_1} \log 3 + 4T, \displaystyle\frac{4}{C_1} \log 3 \Big\}$ then
$$\displaystyle\frac{1}{t}\displaystyle\int_{0}^{t} K(x(s)) \ ds < \max \Big\{  \displaystyle\frac{\eta}{16T},  \displaystyle\frac{\eta}{8T + 2A}    \Big\}<0,$$
for any geodesic $\gamma(t) = (x(t), z(t))$. Therefore by  Theorem \ref{main2}, follows that the geodesic flow of $M = \mathbb{R} \times_{f} \mathbb{S}^1$ is Anosov.

\begin{R} For the family  $g_{a}(x)=ax-\cos x + \sin x$, we have the new family
$$h_{a}(x)=g_{a}''(x)+(g_{a}'(x))^2= 1+a^2  +2a(\sin x +\cos x)+2\sin x \cos x.$$
The functions $h_{a}$ are periodic with period $2\pi$. Moreover,  for $a$ large enough, we have $h_{a}(x)>0$. Thus, in the interval $[0,2\pi]$ the function $h_{a}$ is positive and has a minimum. So, by continuity of the family function $h_{a}$ in the parameter $a$, there exists the  parameter 
$$\overline{a}=\inf \{a\in \re^{+}: h_{a}|_{[0,2\pi]}\geq 0 \ \ \text{and its minimum value is} \ 0 \}.$$
In conclusion, the function $h_{\overline{a}}(x)$ is non-negative and by periodicity, it has a infinite many zeros.\\ Therefore, by \emph{(\ref{e0exam1})}  the surface $M=\re \times_{f_{\overline{a}}}\mathbb{S}$, where $f(x)=e^{g_{\overline{a}}(x)}$ has non-positive curvature with infinite many points of zero curvature and whose geodesic flow is Anosov. Moreover, as the Anosov condition is open, we make a local perturbation of the metric in a point of zero curvature, so that we get positive curvature in a small open set and the new geodesic flow is Anosov.  


\end{R}
\subsection{Example 2}\label{example2}
Finally, in this section we use the warped product to construct a non-compact surface with negative curvature and whose geodesic flow 
is not Anosov. In fact, consider the surface $M = \mathbb{R} \times_{f} \mathbb{S}^1$, where $f(x) = e^{g(x)}$ and $g(x) = e^{-x}$. Observe that the curvature in the point $(x,y)$ is given by 
$$k(x,y) = k(x)= -f''/f = -(g''(x) + (g'(x))^2) = -(e^{-x}+e^{-2x})<0.$$ 

Now consider a ray $\gamma:[0, + \infty) \to M$, where $\gamma(t) = (t, y_0)$ with $y_0 \in \mathbb{S}^1$. Observe that, 

$$ \lim_{t\to +\infty}  \displaystyle\frac{1}{t}\displaystyle\int_{0}^{t} k(\gamma(s)) =0.$$

By Corollary \ref{Cor3'} we have that the geodesic flow of $M$ is not Anosov.\\

{\textbf{Acknowledgments:}} The authors would like to thank to Davi Máximo for useful conversations during the preparation of this work.


\bibliographystyle{alpha}	
\bibliography{Neces_Condition}

\noindent \textbf{\'Italo Dowell Lira Melo}\\
Universidade Federal do Piauí \\ 
Departamento de Matemática-UFPI, Ininga, cep 64049-550 \\
Piau\'i-Brasil \\
E-mail: italodowell@ufpi.edu.br\\
\ \\
\noindent \textbf{Sergio Augusto Roma\~na Ibarra}\\
Universidade Federal do Rio de Janeiro\\
Av. Athos da Silveira Ramos 149, Centro de Tecnologia \ - Bloco C \ - Cidade Universit\'aria Ilha do Fund\~ao, cep 21941-909 \\
Rio de Janeiro-Brasil\\
E-mail: sergiori@im.ufrj.br
\end{document}